\documentclass{amsart}
\usepackage{graphicx}
\usepackage{amsmath}
\usepackage{amsthm}
\usepackage{amsfonts}
\usepackage{amssymb, amscd}

\usepackage[all]{xy}

\newtheorem{thm}{Theorem}[section]
\newtheorem{cor}[thm]{Corollary}
\newtheorem{lem}[thm]{Lemma}
\newtheorem{proposition}[thm]{Proposition}

\theoremstyle{definition}
\newtheorem{definition}[thm]{Definition}
\theoremstyle{remark}
\newtheorem{remark}[thm]{Remark}
\numberwithin{equation}{section}

\newcommand{\K}{\mathbb K}

\newcommand{\A}{\mathcal{A}}
\newcommand{\HH}{\mathcal{H}}

\begin{document}

\title[  Hom-Lie admissible Hom-coalgebras and Hom-Hopf algebras]{  Hom-Lie admissible Hom-coalgebras and Hom-Hopf algebras}%
\author{A. MAKHLOUF and S. SILVESTROV}%
\address{Abdenacer Makhlouf, Universit\'{e} de Haute Alsace,  Laboratoire de Math\'{e}matiques, Informatique et Application,
4, rue des Fr\`{e}res Lumi\`{e}re F-68093 Mulhouse, France}%
\email{Abdenacer.Makhlouf@uha.fr}
\address{Sergei Silvestrov, Centre for Mathematical Sciences,  Lund University, Box
   118, SE-221 00 Lund, Sweden}
\email{ssilvest@maths.lth.se}

\thanks {This work was partially supported by the Crafoord foundation, The Royal Physiographic Society in Lund,
The Swedish Royal Academy of Sciences, The Swedish Foundation of
International Cooperation in Research and High Education (STINT) and
the Liegrits network}
\date{April 10th  2007}
%
\begin{abstract}
The aim of this paper is to generalize the concept of Lie-admissible
coalgebra introduced in \cite{GR} to Hom-coalgebras and to introduce
Hom-Hopf algebras with some properties. These structures are based
on the Hom-algebra structures introduced in \cite{MS}.
\end{abstract}
\maketitle

\section*{Introduction}

In \cite{HLS,LS1,LS2}, the class of quasi-Lie algebras and
subclasses of quasi-hom-Lie algebras and Hom-Lie algebras has been
introduced. These classes of algebras are tailored in a way suitable
for simultaneous treatment of the Lie algebras, Lie superalgebras,
the color Lie algebras and the deformations arising in connection
with twisted, discretized or deformed derivatives and corresponding
generalizations, discretizations and deformations of vector fields
and differential calculus. It has been shown in
\cite{HLS,LS1,LS2,LS3} that the class of quasi-Hom-Lie algebras
contains as a subclass on the one hand the color Lie algebras and in
particular Lie superalgebras and Lie algebras, and on the another
hand various known and new single and multi-parameter families of
algebras obtained using twisted derivations and constituting
deformations and quasi-deformations of universal enveloping algebras
of Lie and color Lie algebras and of algebras of vector-fields. The
main feature of quasi-Lie algebras, quasi-Hom-Lie algebras and
Hom-Lie algebras is that the skew-symmetry and the Jacobi identity
are twisted by several deforming twisting maps and also in quasi-Lie
and quasi-Hom-Lie algebras the Jacobi identity in general contains 6
twisted triple bracket terms.

In the paper \cite{MS}, we provided a different
way for constructing Hom-Lie algebras by
extending the fundamental construction of Lie
algebras from associative algebras via commutator
bracket multiplication. To this end we defined
the notion of Hom-associative algebras
generalizing associative algebras to a situation
where associativity law is twisted, and showed
that the commutator product defined using the
multiplication in a Hom-associative algebra leads
naturally to Hom-Lie algebras. We introduced also
Hom-Lie-admissible algebras and more general
$G$-Hom-associative algebras with subclasses of
Hom-Vinberg and pre-Hom-Lie algebras,
generalizing to the twisted situation
Lie-admissible algebras, $G$-associative
algebras, Vinberg and pre-Lie algebras
respectively, and show that for these classes of
algebras the operation of taking commutator leads
to Hom-Lie algebras as well. We constructed also
all the twistings so that the brackets
$[X_1,X_2]=2 X_2, \ [X_1,X_3]=-2 X_3, \
[X_2,X_3]=X_1$ determine a three dimensional
Hom-Lie algebra. Finally, we provided for a
subclass of twistings, the list of all
three-dimensional Hom-Lie algebras. This list
contains all three-dimensional Lie algebras for
some values of structure constants. The families
of Hom-Lie algebras in these list can be viewed
as deformations of Lie algebras into a class of
Hom-Lie algebras. The notion, constructions and
properties of the enveloping algebras of Hom-Lie
algebras are yet to be properly studied in full
generality. An important progress in this
direction has been made in the recent work by D.
Yau \cite{Yau:EnvLieAlg}.

In the present paper we develop the coalgebra counterpart of the
notions and results of \cite{MS}, extending in particular in the
framework of Hom-associative and Hom-Lie algebras and
Hom-coalgebras, the notions and results on associative and Lie
admissible coalgebras obtained in \cite{GR}. In the first section we
summarize the relevant definitions of Hom-associative algebra,
Hom-Lie algebra, Hom-Leibniz algebra, and define the notions of
Hom-coalgebras and Hom-coassociative coalgebras. In section 2, we
introduce the concept of Hom-Lie admissible Hom-coalgebra, describe
some useful relations between coproduct, opposite coproduct, the
cocommutator defined as their difference, and their $\beta$-twisted
coassociators and $\beta$-twisted co-Jacobi sums. We also introduce
the notion of $G$-Hom-coalgebra for any subgroup $G$ of permutation
group $S_3$. We show that $G$-Hom-coalgebras are Hom-Lie admissible
Hom-coalgebras, and also establish duality based correspondence
between classes of $G$-Hom-coalgebras and $G$-Hom-algebras. The last
section is dedicated to relevant definitions and basic properties of
the Hom-Hopf algebra which generalize the classical Hopf algebra
structure. We also define the module and comodule structure over
Hom-associative algebra or
 Hom-coassociative coalgebra.

\section{Hom-Algebra  and Hom-Coalgebra structures}
  A Hom-algebra
structure is a multiplication on a vector space
where the structure is twisted by a homomorphism.
The structure of Hom-Lie algebra was introduced
by Hartwig, Larson and Silvestrov in \cite{HLS}.
In the following  we summarize the definitions of
Hom-associative, Hom-Leibniz, and
Hom-Lie-admissible algebraic structures
introduced in \cite{MS} and generalizing the well
known associative, Leibniz and Lie-admissible
algebras. By dualization of Hom-associative
algebra we define the Hom-coassociative coalgebra
structure.
\subsection{Hom-algebra structures}
Let $\mathbb{K}$ be an algebraically closed field of characteristic
$0$ and $V$ be a vector space over $\mathbb{K}$.
\begin{definition}
A \emph{Hom-associative algebra} over $V$ is a linear map $\mu :
V\otimes V \rightarrow V$ and a homomorphism
$\alpha$ satisfying
\begin{equation}\label{Hom-ass}
\mu(\alpha(x)\otimes \mu (y\otimes z))= \mu (\mu
(x\otimes y)\otimes \alpha (z)).
\end{equation}
\end{definition}
The Hom-associativity condition \eqref{Hom-ass}
may be expressed by the following commutative
diagram.
$$
\begin{array}{ccc}
V\otimes V\otimes V & \stackrel{\mu \otimes \alpha}{\longrightarrow
} & V\otimes
V \\
\ \quad \left\downarrow ^{\alpha\otimes \mu }\right. &  & \quad
\left\downarrow
^\mu \right. \\
V\otimes V & \stackrel{\mu }{\longrightarrow } & V
\end{array}
$$

The Hom-associative algebra is unital if there exists a homomorphism
$\eta :\mathbb{K}\rightarrow V$ such that the following diagrams are
commutative

$$\begin{array}{ccccc} \K \otimes V &
\stackrel{\eta \otimes id}{\longrightarrow } & V\otimes V &
\stackrel{id\otimes \eta }{\longleftarrow } & V\otimes \K \\
& \searrow ^{\cong } & \quad \left\downarrow ^\mu \right. & \swarrow
^{\cong
} &  \\
&  & V &  &
\end{array}
$$
In the language of Hopf algebra, a Hom-associative algebra $\A$ is a
quadruple $(V,\mu,\alpha,\eta)$ where $V$ is the vector space, $\mu$
is the Hom-associative multiplication, $\alpha$ is the twisting
homomorphism and $\eta$ is the unit.

 Let $\left( V,\mu ,\alpha,\eta \right) $ and $\left( V^{\prime },\mu
^{\prime },\alpha^{\prime },\eta ^{\prime }\right) $ be two
Hom-associative algebras. A linear map $f\
:V\rightarrow V^{\prime }$ is a morphism of Hom- associative algebras if%
$$
\mu ^{\prime }\circ (f\otimes f)=f\circ \mu
\qquad \text{,}\qquad f\circ \eta =\eta ^{\prime
}\qquad \text{and}\qquad f\circ
\alpha=\alpha^{\prime }\circ f.
$$

In particular, $\left( V,\mu,\alpha ,\eta \right) $ and $\left(
V,\mu ^{\prime },\alpha^{\prime },\eta ^{\prime }\right) $ are
isomorphic if there exists a
bijective linear map $f\ $such that%
$$
\mu =f^{-1}\circ \mu ^{\prime }\circ (f\otimes
f)\qquad \text{,}\qquad \eta =f^{-1}\circ \eta
^{\prime }\qquad \text{and}\qquad \alpha=
f^{-1}\circ \alpha^{\prime }\circ f.
$$

The tensor product of two Hom-associative algebras $\left( V_1,\mu
_1,\alpha_1 ,\eta_1 \right) $ and $\left( V_2,\mu_2,\alpha_2 ,\eta_2
\right) $ is defined in an obvious way by the Hom-associative
algebra $\left( V_1\otimes V_2,\mu _1 \otimes \mu_2,\alpha_1 \otimes
\alpha_2 ,\eta_1 \otimes\eta_2 \right) $.

The Hom-Lie algebras were initially introduced by Hartwig, Larson
and Silvestrov in \cite{HLS} motivated initially by examples of
deformed Lie algebras coming from twisted discretizations of vector
fields.

\begin{definition}
A \emph{Hom-Lie algebra} is a triple $(V, [\cdot, \cdot], \alpha)$
consisting of
 a linear space $V$, bilinear map $[\cdot, \cdot]: V\times V \rightarrow V$ and
 a linear space homomorphism $\alpha: V \rightarrow V$
 satisfying
$$\begin{array}{c} [x,y]=-[y,x] \quad {\text{(skew-symmetry)}} \\{}
\circlearrowleft_{x,y,z}{[\alpha(x),[y,z]]}=0 \quad
{\text{(Hom-Jacobi condition)}}
\end{array}$$
for all $x, y, z$ from $V$, where $\circlearrowleft_{x,y,z}$ denotes
summation over the cyclic permutation on $x,y,z$.
\end{definition}
In a similar way we have the following definition of Hom-Leibniz
algebra.
\begin{definition}
A \emph{Hom-Leibniz algebra} is a triple $(V, [\cdot, \cdot],
\alpha)$ consisting of a linear space $V$, bilinear map $[\cdot,
\cdot]: V\times V \rightarrow V$ and a homomorphism $\alpha: V
\rightarrow V$  satisfying
\begin{equation} \label{Leibnizalgident}
 [[x,y],\alpha(z)]=[[x,z],\alpha (y)]+
 [\alpha(x),[y,z]].
\end{equation}
\end{definition}
Note that if a Hom-Leibniz algebra is skewsymmetric then it is a
Hom-Lie algebra.

\subsection{Hom-Coalgebra structures}
\begin{definition}
A \emph{Hom-coassociative coalgebra} is a quadruple $\left( V,\Delta
,\beta, \varepsilon \right) $ where $V$ is a $\K$-vector space and
$$\Delta :V\rightarrow V\otimes V,\quad \beta :V\rightarrow V \
\text{ and } \ \varepsilon :V\rightarrow \K $$
are linear maps satisfying the following
conditions:

\begin{itemize}
\item[(C1)] $\qquad \left( \beta\otimes \Delta
\right) \circ \Delta =\left( \Delta \otimes
\beta\right) \circ \Delta $
\item[(C2)] $\qquad \left( id\otimes \varepsilon
\right) \circ \Delta =id\qquad $  and \qquad
$\left( \varepsilon \otimes id\right) \circ
\Delta =id$.
\end{itemize}
\end{definition}

The condition (C1) expresses the
Hom-coassociativity  of the comultiplication
$\Delta $. Also,  it is equivalent to the
following commutative diagram:
$$
\begin{array}{ccc}
 V & \stackrel{\Delta}{\longrightarrow
} & V\otimes
V \\
\ \quad \left\downarrow ^{ \Delta}\right. &  & \quad \left\downarrow
^{\beta\otimes\Delta}  \right. \\
V\otimes V & \stackrel{\Delta \otimes \beta}{\longrightarrow } &
V\otimes V\otimes V
\end{array}
$$

The condition (C2) expresses that $\varepsilon $
is the counit which is also  equivalent to the
following commutative diagrams:

$$\begin{array}{ccccc} \K \otimes V &
\stackrel{\varepsilon \otimes id_V}{\longleftarrow } & V\otimes V &
\stackrel{id\otimes \varepsilon }{\longrightarrow } & V\otimes \K \\
& \nwarrow ^{\cong } & \quad \left\uparrow ^\Delta \right. &
\nearrow^{\cong
} &  \\
&  & V &  &
\end{array}
$$

Let $\left( V,\Delta ,\beta,\varepsilon \right) $
and $\left( V^{\prime },\Delta ^{\prime
},\beta^{\prime }, \varepsilon ^{\prime }\right)
$ be two Hom-coassociative coalgebras. A
linear map  $f\ :V\rightarrow V^{\prime }$ is a morphism of Hom-coassociative coalgebras if%
$$
(f\otimes f)\circ \Delta =\Delta ^{\prime }\circ f \quad \text{,}
 \quad\varepsilon =\varepsilon ^{\prime }\circ f \quad \text{ and}\quad f\circ\beta= \beta^{\prime }\circ f.$$

If $V=V^{\prime }$ the previous Hom-coassociative
coalgebras are isomorphic if there exists a
bijective linear map $f\ :V\rightarrow V$ such
that $$ \Delta ^{\prime }=(f\otimes f)\circ
\Delta \circ f^{-1}\quad \text{,} \quad
\varepsilon ^{\prime }=\varepsilon \circ f^{-1}
\quad \text{ and}\quad \beta=f^{-1}\circ
\beta^{\prime }\circ f. $$

In the sequel, we call \emph{Hom-coalgebra} a
triple $(V,\Delta, \beta )$ where $V$ is a
$\K$-vector space, $\Delta$ is a comultiplication
not necessarily coassociative or
Hom-coassociative, that is a linear map $\Delta
: V\rightarrow V\otimes V$, and $\beta$ is a
linear map $\beta : V\rightarrow V.$

\section{Hom-Lie admissible Hom-Coalgebras}

Let $\mathbb{K}$ be an algebraically closed field
of characteristic $0$ and $V$ be a vector space
over $\mathbb{K}$. Let $(V,\Delta, \beta )$ be a
Hom-coalgebra where $\Delta  : V\rightarrow
V\otimes V$ and $\beta : V\rightarrow V$ are
linear maps and $\Delta$  is not necessarily
coassociative or Hom-coassociative.

By a  \emph{$\beta$-coassociator} of $\Delta$ we call a linear map
$\textbf{c}_\beta (\Delta)$ defined by
$$\textbf{c}_\beta (\Delta):= \left(
\Delta \otimes \beta\right) \circ \Delta-
\left( \beta\otimes \Delta \right) \circ \Delta.
$$
Let $\mathcal{S}_3$ be the symmetric group of order 3. Given
$\sigma\in \mathcal{S}_3 $, we define a linear map
$$\Phi _ \sigma \ : V^{\otimes 3}\longrightarrow V^{\otimes 3}
$$
by
$$\Phi _ \sigma (x_1\otimes x_2 \otimes x_3 )= x_{\sigma ^{-1}(1)}  \otimes x_{\sigma
^{-1}(2)}\otimes x_{\sigma ^{-1}(3)}.
$$
Recall that $\Delta^{op}=\tau\circ \Delta$ where $\tau$ is the usual
flip that is $\tau (x\otimes y)= y \otimes x$.
\begin{definition} \label{def:HomLieadmissibleHomcoalg}
  A triple
$( V, \Delta, \beta)$   is a \emph{Hom-Lie admissible Hom-coalgebra}
if the linear map
$$\Delta_L : V\longrightarrow V\otimes V
$$
defined by $\Delta _L = \Delta - \Delta^{op} $,  is a Hom-Lie
coalgebra multiplication, that is the following condition is
satisfied
\begin{eqnarray}\label{CoHomLie}
  \textbf{c}_\beta (\Delta_L)+\Phi_{(213)} \circ \textbf{c}_\beta (\Delta_L)+
   \Phi_{(231)} \circ \textbf{c}_\beta (\Delta_L) =
  0
\end{eqnarray}
where $(213)$ and $(231)$ are the two cyclic permutations of order 3
in $\mathcal{S}_3$.
\end{definition}
\begin{remark}
 Since $\Delta _L = \Delta - \Delta^{op} $,
 the equality $ \Delta_L ^{op}= - \Delta_L$ holds.
 \end{remark}

\begin{lem}
Let $(V,\Delta, \beta )$ be a Hom-coalgebra where $\Delta  :
V\rightarrow V\otimes V$ and $\beta : V\rightarrow V$ are linear
maps and $\Delta$  is not necessarily coassociative or
Hom-coassociative, then the following relations are true
\begin{eqnarray}
\textbf{c}_\beta (\Delta^{op})&=&-\Phi_{(13)} \circ \textbf{c}_\beta
(\Delta) \\
(\beta \otimes \Delta^{op})\circ \Delta &=& \Phi_{(13)} \circ
(\Delta \otimes \beta) \circ \Delta^{op}
 \\
 (\beta \otimes \Delta)\circ \Delta ^{op} &=& \Phi_{(13)} \circ
(\Delta ^{op} \otimes \beta) \circ \Delta
 \\
 (\Delta \otimes \beta)\circ \Delta ^{op} &=& \Phi_{(213)} \circ
( \beta\otimes \Delta ) \circ \Delta
 \\
 (\Delta  ^{op}\otimes \beta)\circ \Delta &=& \Phi_{(12)} \circ
( \Delta\otimes \beta ) \circ \Delta .
\end{eqnarray}
\end{lem}
\begin{lem}
The $\beta$-coassociator of $\Delta_{L}$ is expressed using $\Delta$
and $\Delta ^{op}$ as follows:
\begin{eqnarray}
\label{coassDeltaLviaDeltaDeltaOp} \textbf{c}_\beta
(\Delta_{L})&=&\textbf{c}_\beta (\Delta)+\textbf{c}_\beta
(\Delta^{op})\\ \nonumber && -(\Delta \otimes \beta)\circ \Delta
^{op}-(\Delta  ^{op}\otimes \beta)\circ \Delta +\\ \nonumber  &&
\Phi_{(13)} \circ (\Delta  \otimes \beta) \circ
\Delta^{op}+\Phi_{(13)} \circ ( \Delta ^{op} \otimes \beta )
\circ \Delta \\
\label{coassDeltaLviaDelta} &=&\textbf{c}_\beta
(\Delta)-\Phi_{(13)} \circ \textbf{c}_\beta (\Delta)\\
\nonumber && -\Phi_{(213)} \circ (\beta \otimes
\Delta)\circ \Delta -\Phi_{(12)} \circ (\Delta
\otimes \beta)\circ \Delta \\ \nonumber &&
+\Phi_{(23)} \circ (\beta \otimes \Delta) \circ
\Delta +\Phi_{(231)} \circ ( \Delta \otimes \beta
) \circ \Delta .
\end{eqnarray}
\end{lem}

\begin{proposition} \label{prop:coJacobiSumA3viaS3}
Let  $( V, \Delta, \beta)$ be  a Hom-coalgebra. Then one has
\begin{equation} \label{coJacobiSumA3viaS3}
\textbf{c}_\beta (\Delta_L)+\Phi_{(213)} \circ \textbf{c}_\beta
(\Delta_L)+
   \Phi_{(231)} \circ \textbf{c}_\beta (\Delta_L) = 2 \sum_{\sigma \in \mathcal{S}_3} {(-1)^{\epsilon (\sigma)}\Phi
_\sigma \circ \textbf{c}_\beta (\Delta)}
\end{equation}
where $(-1)^{\epsilon (\sigma)}$ is the signature of
the permutation $\sigma$.
\end{proposition}
\begin{proof}
By \eqref{coassDeltaLviaDelta}  and multiplication rules in the
group $S_3$, it follows that
\begin{eqnarray} \nonumber \Phi_{(213)} \circ
\textbf{c}_\beta (\Delta_L) &=& \Phi_{(213)} \circ \textbf{c}_\beta
(\Delta)
-\Phi_{(213)}\circ \Phi_{(13)} \circ \textbf{c}_\beta (\Delta)\\
\nonumber && -\Phi_{(213)}\circ\Phi_{(213)} \circ (\beta \otimes
\Delta)\circ \Delta -\Phi_{(213)}\circ\Phi_{(12)} \circ (\Delta
\otimes \beta)\circ \Delta \\ \nonumber &&
+\Phi_{(213)}\circ\Phi_{(23)} \circ (\beta \otimes \Delta) \circ
\Delta +\Phi_{(213)}\circ\Phi_{(231)} \circ ( \Delta \otimes \beta )
\circ \Delta  \\
\label{Phi(213)coassDeltaLviaDelta} &=& \Phi_{(213)} \circ
\textbf{c}_\beta (\Delta)
-\Phi_{(12)}\circ \textbf{c}_\beta (\Delta)\\
\nonumber && -\Phi_{(231)} \circ (\beta \otimes \Delta)\circ \Delta
-\Phi_{(23)} \circ (\Delta \otimes \beta)\circ \Delta \\ \nonumber
&& +\Phi_{(13)} \circ (\beta \otimes \Delta) \circ \Delta + ( \Delta
\otimes \beta ) \circ \Delta, \\
\nonumber \Phi_{(231)} \circ \textbf{c}_\beta (\Delta_L) &=&
\Phi_{(231)} \circ \textbf{c}_\beta (\Delta)
-\Phi_{(231)}\circ \Phi_{(13)} \circ \textbf{c}_\beta (\Delta)\\
\nonumber && -\Phi_{(231)}\circ\Phi_{(213)} \circ (\beta \otimes
\Delta)\circ \Delta -\Phi_{(231)}\circ\Phi_{(12)} \circ (\Delta
\otimes \beta)\circ \Delta \\ \nonumber &&
+\Phi_{(231)}\circ\Phi_{(23)} \circ (\beta \otimes \Delta) \circ
\Delta +\Phi_{(231)}\circ\Phi_{(231)} \circ ( \Delta \otimes \beta )
\circ \Delta  \\
\label{Phi(231)coassDeltaLviaDelta} &=& \Phi_{(231)} \circ
\textbf{c}_\beta (\Delta)
-\Phi_{(23)}\circ \textbf{c}_\beta (\Delta)\\
\nonumber && -(\beta \otimes \Delta)\circ \Delta -\Phi_{(13)} \circ
(\Delta \otimes \beta)\circ \Delta \\ \nonumber && +\Phi_{(12)}
\circ (\beta \otimes \Delta) \circ \Delta + \Phi_{(213)} \circ
(\Delta \otimes \beta ) \circ \Delta.
\end{eqnarray}
After summing up the equalities \eqref{coassDeltaLviaDelta},
\eqref{Phi(213)coassDeltaLviaDelta} and
\eqref{Phi(231)coassDeltaLviaDelta} the terms on the right hand
sides may be pairwise combined into the terms of the form
${(-1)^{\epsilon (\sigma)}\Phi _\sigma \circ \textbf{c}_\beta
(\Delta)}$ with each one being present in the sum twice for all
$\sigma \in S_3$.
\end{proof}

Definition \ref{def:HomLieadmissibleHomcoalg} together with
\eqref{coJacobiSumA3viaS3} yields the following corollary.

\begin{cor}
A triple $( V, \Delta, \beta)$ is a Hom-Lie admissible Hom-coalgebra
if and only if
$$\sum_{\sigma \in \mathcal{S}_3}{(-1)^{\epsilon (\sigma)}\Phi
_\sigma \circ \textbf{c}_\beta (\Delta)}=0
$$
where $(-1)^{\epsilon (\sigma)}$ is the signature of the permutation
$\sigma$.
\end{cor}

\subsection{$G$-Hom-Coalgebra structures}
In this section we introduce, as in the multiplication case, the
notion of \emph{$G$-Hom-coalgebra} where $G$ is a subgroup of the
symmetric group $\mathcal{S}_3$.
\begin{definition}
Let $G$ be a subgroup of the symmetric group $\mathcal{S}_3$, A
Hom-coalgebra $( V, \Delta, \beta)$ is called
\emph{$G$-Hom-coalgebra} if
\begin{equation}\label{admi}
\sum_{\sigma \in G}{(-1)^{\epsilon (\sigma)}\Phi _\sigma \circ
\textbf{c}_\beta (\Delta)}=0
\end{equation}
where  $(-1)^{\varepsilon ({\sigma})}$ is the signature of the
permutation $\sigma$.
\end{definition}

\begin{proposition}
Let $G$ be a subgroup of the permutations group $\mathcal{S}_3$.
Then any $G$-Hom-Coalgebra $( V, \Delta, \beta)$ is a Hom-Lie
admissible Hom-coalgebra.
\end{proposition}
\begin{proof}
The skew-symmetry follows straightaway from the
definition. Take the set of conjugacy classes
$\{g G\}_{g\in I}$  where $I\subseteq G$, and for
any $\sigma_1, \sigma_2\in I,\sigma_1 \neq
\sigma_2 \Rightarrow \sigma_1 G\bigcap \sigma_1 G
=\emptyset$. Then

$$\sum_{\sigma\in \mathcal{S}_3}{{(-1)^{\epsilon (\sigma)}}\Phi _\sigma \circ
\textbf{c}_\beta (\Delta)}=\sum_{\sigma_1\in
I}{\sum_{\sigma_2\in \sigma_1 G}{(-1)^{\epsilon
(\sigma)}\Phi _\sigma \circ \textbf{c}_\beta
(\Delta)}}=0.$$
\end{proof}

The subgroups of $\mathcal{S}_3$ are
$$G_1=\{Id\}, ~G_2=\{Id,\tau_{1 2}\},~G_3=\{Id,\tau_{2
3}\},$$ $$~G_4=\{Id,\tau_{1 3}\},~G_5=A_3 ,~G_6=\mathcal{S}_3,$$
where $A_3$ is the alternating group and where $\tau_{ij}$ is the
transposition between $i$ and $j$.

We obtain the following type of Hom-Lie-admissible Hom-coalgebras.
\begin{itemize}
\item The  $G_1$-Hom-coalgebras  are the Hom-associative
coalgebras defined above.

\item The  $G_2$-Hom-coalgebras satisfy the condition
\begin{equation}\nonumber
\textbf{c}_\beta (\Delta)+
\Phi_{(12)}\textbf{c}_\beta (\Delta)=0.
\end{equation}

\item The  $G_3$-Hom-coalgebras satisfy the condition
\begin{equation}\nonumber
\textbf{c}_\beta (\Delta)+
\Phi_{(23)}\textbf{c}_\beta (\Delta)=0.
\end{equation}

\item The  $G_4$-Hom-coalgebras satisfy the condition
\begin{equation}\nonumber
\textbf{c}_\beta (\Delta)+
\Phi_{(13)}\textbf{c}_\beta (\Delta)=0.
\end{equation}
\item The  $G_5$-Hom-coalgebras satisfy the condition
\begin{eqnarray}\nonumber
\textbf{c}_\beta (\Delta)+
\Phi_{(213)}\textbf{c}_\beta (\Delta)+
\Phi_{(231)}\textbf{c}_\beta (\Delta)=0.
\end{eqnarray}\nonumber
If the product $\mu$ is skewsymmetric then the previous condition is
exactly the Hom-Jacobi identity.
\item The  $G_6$-Hom-coalgebras are the Hom-Lie-admissible
coalgebras.
\end{itemize}

The $G_2$-Hom-coalgebras may be called
Vinberg-Hom-coalgebra and $G_3$-Hom-coalgebras
may be called preLie-Hom-coalgebras.

\begin{definition}
A \emph{Vinberg-Hom-coalgebra} is a triple $(V, \Delta, \beta)$
consisting of a linear space $V$, a linear map $\mu: V \rightarrow
V\times V$ and a homomorphism $\beta$ satisfying
\begin{equation}\nonumber
\textbf{c}_\beta (\Delta)+
\Phi_{(12)}\textbf{c}_\beta (\Delta)=0.
\end{equation}
\end{definition}

\begin{definition}
A \emph{preLie-Hom-coalgebra} is a triple $(V, \Delta, \beta)$
consisting of a linear space $V$, a linear map $\mu: V \rightarrow
V\times V$ and a homomorphism $\beta$ satisfying
\begin{equation}\nonumber
\textbf{c}_\beta (\Delta)+
\Phi_{(23)}\textbf{c}_\beta (\Delta)=0.
\end{equation}
\end{definition}

More generally, by dualization we have a
correspondence between $G$-Hom-associative
algebras introduced in \cite{MS} and
$G$-Hom-coalgebras for a subgroup $G$ of
$\mathcal{S}_3$.

Let $G$ be a subgroup  of $\mathcal{S}_3$ and $(V, \mu, \alpha)$ be
a $G$-Hom-associative  algebra that is $\mu : V\otimes V \rightarrow
V$ and $ \alpha : V \rightarrow V$ are linear maps and the following
condition is satisfied
\begin{equation}\label{G_ass}
\sum_{\sigma\in G}{(-1)^{\varepsilon ({\sigma})}a_{\alpha,\mu}\circ
\Phi_\sigma}=0.
\end{equation}
where $a_{\alpha,\mu}$ is the $\alpha$-associator that is
$a_{\alpha,\mu}= \mu \circ(\mu  \otimes \alpha)
-\mu\circ(\alpha\otimes \mu) $

Setting $$(\mu \otimes \alpha)_G =\sum_{\sigma\in
G}{(-1)^{\varepsilon ({\sigma})}(\mu \otimes
\alpha)\circ \Phi_\sigma}\  \text{ and } \
(\alpha \otimes \mu )_G =\sum_{\sigma\in
G}{(-1)^{\varepsilon ({\sigma})}(\alpha \otimes
\mu )\circ \Phi_\sigma}$$ the condition
\eqref{G_ass} is equivalent to the following
commutative diagram

$$
\begin{array}{ccc}
V\otimes V\otimes V & \stackrel{(\mu \otimes
\alpha)_G}{\longrightarrow } & V\otimes
V \\
\ \quad \left\downarrow ^{(\alpha \otimes \mu )_G}\right. &  & \quad
\left\downarrow
^\mu \right. \\
V\otimes V & \stackrel{\mu }{\longrightarrow } & V
\end{array}
$$

By the dualization of the square one may obtain the following
commutative diagram

$$
\begin{array}{ccc}
 V & \stackrel{\Delta}{\longrightarrow
} & V\otimes
V \\
\ \quad \left\downarrow ^{ \Delta}\right. &  & \quad \left\downarrow
^{(\beta\otimes\Delta)_G}  \right. \\
V\otimes V & \stackrel{(\Delta \otimes \beta)_G}{\longrightarrow } &
V\otimes V\otimes V
\end{array}
$$

where $$ (\beta\otimes\Delta)_G=\sum_{\sigma\in
G}{(-1)^{\varepsilon ({\sigma})}\Phi_\sigma \circ
(\beta \otimes\Delta )}  \ \text{ and }\ (\Delta
\otimes \beta)_G=\sum_{\sigma\in
G}{(-1)^{\varepsilon ({\sigma})}\Phi_\sigma \circ
(\Delta\otimes\beta)}.
$$

The previous commutative diagram expresses that $(V, \Delta, \beta)$
is a $G$-Hom-coalgebra. More precisely we have the following
connection between $G$-Hom-coalgebras and $G$-Hom-associative
algebras.

\begin{proposition}
Let  $\left( V,\Delta ,\beta \right) $ be a $G$-Hom-coalgebra where
$G$ is a subgroup of $\mathcal{S}_3$. Its dual vector space
$V^\star$ is provided with a $G$-Hom-associative algebra  $\left(
V^\star,\Delta^\star ,\beta^\star \right) $ where $\Delta^\star
,\beta^\star$ are the transpose map.
\end{proposition}
\begin{proof}
Let  $\left( V,\Delta , \beta \right) $ be a $G$-Hom-coalgebra. Let
$V^{\star}$ be the  dual space of $V$ ($V^{\star}=Hom(V,\K)$).

Consider the map
$$\lambda_n : (V^{\star})^{\otimes n}\longrightarrow (V^{\star})^{\otimes n}
$$
$$ f_1\otimes \cdots \otimes f_n \longrightarrow  \lambda_n (f_1\otimes \cdots \otimes f_n)$$

such that for $v_1 \otimes \cdots \otimes  v_n\in V^{\otimes n}$
$$\lambda_n (f_1\otimes \cdots \otimes f_n)(v_1 \otimes \cdots
\otimes v_n)=f_1(v_1) \otimes \cdots \otimes f_n(v_n)$$

and set $$ \mu :=\Delta^{\star}\circ \lambda_2 \quad \quad \alpha :=
 \beta ^{ \ast}$$
where the star $\star$ denotes the transpose linear map. Then, the
quadruple $\left( V^{\star},\mu ,\eta , \alpha\right) $ is a
 $G$-Hom-associative algebra. Indeed,  $\mu (f_1,f_2)=\mu_\K \circ
 \lambda_2 (f_1\otimes f_2) \circ \Delta $ where $\mu_\K$ is the multiplication of $\K$ and $f_1,f_2 \in
 V^{\ast}$. One has
\begin{eqnarray}\mu \circ ( \mu \otimes \alpha) (f_1 \otimes f_2
\otimes f_3 )&=& \mu
 (\mu (f_1 \otimes f_2) \otimes \alpha (f_3)) \nonumber \\ &=&\mu_\K \circ \lambda_2(\mu (f_1 \otimes f_2) \otimes \alpha
 (f_3))\circ \Delta \nonumber \\ &=& \mu_\K \circ \lambda_2 (\lambda_2( (f_1\otimes f_2) \circ
 \Delta)\otimes \alpha
 (f_3))\circ \Delta \nonumber  \\ &=&
 \mu_\K \circ (\mu_\K \otimes id ) \circ \lambda_3 (f_1 \otimes f_2
 \otimes f_3) \circ (\Delta \otimes \beta )\circ
 \Delta . \nonumber
\end{eqnarray}
Similarly
$$\mu \circ ( \alpha \otimes \mu ) (f_1 \otimes f_2
\otimes f_3 )= \mu_\K \circ ( id\otimes  \mu_\K)
\circ \lambda_3 (f_1 \otimes f_2 \otimes f_3)
\circ ( \beta\otimes\Delta  )\circ \Delta.$$
 Using the associativity and the commutativity of $\mu_\K$, the
$\alpha$-associator may be written as
 $$a_{\alpha , \mu}=\mu_\K \circ ( id\otimes  \mu_\K) \circ \lambda_3 (f_1
\otimes f_2
 \otimes f_3)\circ(  ( \Delta\otimes\beta  )\circ
 \Delta- ( \beta\otimes\Delta  )\circ
 \Delta).
 $$

 Then we have the following connection between the
 $\alpha$-associator and $\beta$-coassociator
$$a_{\alpha , \mu}=\mu_\K \circ ( id\otimes  \mu_\K) \circ \lambda_3 (f_1
\otimes f_2
 \otimes f_3)\circ \textbf{c}_\beta (\Delta).
 $$
 Therefore if $\left( V,\Delta ,\beta \right) $ is a $G$-Hom-coalgebra, then the   $\left(
V^\star,\Delta^\star ,\beta^\star \right) $ is a $G$-Hom-associative
algebra.

\end{proof}

\begin{proposition}
Let  $\left( V,\mu ,\alpha \right) $ be a finite dimensional
$G$-Hom-associative algebra where $G$ is a subgroup of
$\mathcal{S}_3$. Its dual vector space $V^\star$ is provided with a
$G$-Hom-coalgebra  $\left( V^\star,\mu^\star ,\alpha^\star \right)
$, where $\mu^\star ,\alpha^\star$ are the transpose map.
 \end{proposition}
\begin{proof}
Let $\A=\left( V,\mu,\alpha  \right) $ be a $n$-dimensional
Hom-associative algebra ($n$ finite). Let $\{e_1,\cdots,e_n\}$ be a
basis of $V$ and $\{e^*_1,\cdots,e^*_n\}$ be the dual basis. Then
$\{e^*_i \otimes e^*_j\}_{i,j}$ is a basis of $\A^\star \otimes
\A^\star$. The comultiplication $\Delta=\mu^\star$ on  $\A^\ast$ is
defined for $f\in \A^\star$ by
$$\Delta (f)=\sum_{i,j=1}^{n}{f(\mu(e_i \otimes e_j))\ e^*_i \otimes e^*_j }
$$
Setting $\mu(e_i \otimes e_j)=\sum_{k=1}^n {C_{ij}^k e_k}$ and
$\alpha(e_i)=\sum_{k=1}^n {\alpha_{i}^k e_k}$, then $\Delta
(e^*_k)=\sum_{i,j=1}^n{C_{ij}^k\ e^*_i \otimes e^*_j } $ and
$\beta(e_i)=\alpha^\star(e_i) =\sum_{k=1}^n {\alpha_{k}^i e_k}$.

The condition \eqref{admi} of $G$-Hom-coassociativity of  $\Delta$,
applied to any element $e^*_k$ of the basis, is equivalent to
$$\sum_{p,q,s=1}^{n}{\sum_{\sigma\in G}{(-1)^{\epsilon (\sigma)} (\sum_{i,j=1}^{n}{\alpha_s^j C_{ij}^k C_{pq}^i -
\alpha _p^i C_{ij}^k C_{qs}^j})}e^*_{\sigma ^{-1}(p)} \otimes
e^*_{\sigma ^{-1}(q)}\otimes e^*_{\sigma {-1}(s)}}=0
$$
Therefore $\Delta$ is $G$-Hom-coassociative if for  any
$p,q,s,k\in\{1,\cdots, n\}$ one has
$$\sum_{\sigma\in G}{(-1)^{\epsilon (\sigma)} (\sum_{i,j=1}^{n}{\alpha_s^j C_{ij}^k C_{pq}^i -
\alpha _p^i C_{ij}^k C_{qs}^j})}=0$$ The previous system is exactly
the condition \eqref{G_ass} of $G$-Hom-associativity of $\mu$,
written on $e_{p'}\otimes e_{q'}\otimes e_{s'}$ and setting
$p=\sigma (p'),\ q=\sigma (q'),\ s=\sigma (s')$.
\end{proof}
\begin{cor}
The dual vector space of a Hom-coassociative
coalgebra $\left( V,\Delta ,\beta,\varepsilon
\right) $ is a  Hom-associative algebra $\left(
V^\star,\Delta^\star
,\beta^\star,\varepsilon^\star \right) $, where
$V^\star$ is the dual vector space and the star
for the linear maps denotes the transpose map.
The dual vector space of finite-dimensional
Hom-associative algebra is a Hom-coassociative
coalgebra.
\end{cor}
\begin{proof}It is a particular case of the previous Propositions ($G=G_1$).
\end{proof}
\section{Hom-Hopf algebras}
In this section, we introduce a generalization of Hopf algebras and
show some relevant properties of the new structure. We also define
the module and comodule structure over Hom-associative algebra or
 Hom-coassociative coalgebra. Let $\mathbb{K}$ be an
algebraically closed field of characteristic $0$ and $V$ be a vector
space over $\mathbb{K}$.
\begin{definition}
A \emph{Hom-bialgebra} is a quintuple $\left( V,\mu ,\alpha,\eta
,\Delta ,\beta,\varepsilon \right) $ where

(B1)\qquad $\left( V,\mu ,\alpha,\eta \right) $ is a Hom-associative
algebra

(B2) $\qquad \left( V,\Delta ,\beta,\varepsilon \right) $ is a
Hom-coassociative coalgebra

(B3)\qquad The linear maps $\Delta $ and $\varepsilon $ are
morphisms of  algebras $\left( V,\mu,\alpha ,\eta \right) $.

\end{definition}
\begin{remark}
The condition (B3) could be expressed by the
following system:

$$
\left\{
\begin{array}{l}
\Delta \left( e_1\right) =e_1\otimes e_1\qquad \text{where}\
e_1=\eta \left( 1\right) \\ \Delta \left( \mu(x\otimes
y)\right)=\Delta \left( x \right) \bullet \Delta \left( y\right)
=\sum_{(x)(y)}\mu(x^{(1)}\otimes y^{\left( 1\right) })\otimes \mu(
x^{\left( 2\right)
}\otimes y^{\left( 2\right) })\qquad \\
\varepsilon \left( e_1\right) =1 \\
\varepsilon \left( \mu(x\otimes y)\right) =\varepsilon \left(
x\right) \varepsilon \left( y\right)
\end{array}
\right.
$$
where the bullet $\bullet$ denotes the multiplication on tensor
product and by  using the Sweedler's notation $\Delta \left(
x\right) =\sum_{(x)}x^{(1)}\otimes x^{\left( 2\right) }$. If there
is no ambiguity we denote the multiplication by a dot.
\end{remark}

\begin{remark}
One can consider a more restrictive definition where  linear maps
$\Delta $ and $\varepsilon $ are morphisms of Hom-associative
algebras that is the condition (B3) becomes equivalent to
$$
\left\{
\begin{array}{l}
\Delta \left( e_1\right) =e_1\otimes e_1\qquad \text{where}\
e_1=\eta \left( 1\right) \\ \Delta \left( \mu(x\otimes
y)\right)=\Delta \left( x \right) \bullet \Delta \left( y\right)
=\sum_{(x)(y)}\mu(x^{(1)}\otimes y^{\left( 1\right) })\otimes \mu(
x^{\left( 2\right)
}\otimes y^{\left( 2\right) })\qquad \\
\varepsilon \left( e_1\right) =1 \\
\varepsilon \left( \mu(x\otimes y)\right) =\varepsilon \left(
x\right) \varepsilon \left( y\right)\\
\Delta \left( \alpha(x)\right)=\sum_{(x)}{\alpha(x^{(1)})\otimes
\alpha( x^{\left( 2\right) })}\\
\varepsilon\circ \alpha \left( x\right) =\varepsilon \left( x\right)
\end{array}
\right.
$$
\end{remark}
 Given a Hom-bialgebra $\left( V,\mu ,\alpha,\eta ,\Delta,\beta ,\varepsilon
\right) $, we show that the vector space $Hom \left( V,V \right)$
with the multiplication given by the convolution product carries a
structure of Hom-algebra.
\begin{proposition}
Let  $\left( V,\mu ,\alpha,\eta ,\Delta,\beta ,\varepsilon \right) $
be a Hom-bialgebra. Then the algebra $Hom \left( V,V \right)$ with
the multiplication given by the convolution product defined by
$$ f \star g=\mu \circ \left( f\otimes g \right) \circ\Delta $$
and the unit being $\eta \circ \epsilon$ is a Hom-associative
algebra with the homomorphism map defined by $\gamma (f)=\alpha
\circ f \circ \beta$.
\end{proposition}
\begin{proof}
Let $f,g,h\in Hom \left( V,V \right)$.
\begin{eqnarray*}
  \gamma (f) *( g*h)) &=& \mu \circ \left( \gamma (f)\otimes ( g*h) \right) \Delta  \\
   &=& \mu \circ \left( \gamma (f)\otimes ( \mu \circ \left( g\otimes h \right)\circ \Delta) \right) \Delta \\
   &=& \mu \circ \left( \alpha \otimes
   \mu \right)\circ
   \left( f  \otimes g\otimes h \right)
   \circ \left( \beta \otimes \Delta) \right)
   \Delta .
\end{eqnarray*}
Similarly
$$(f * g)*\gamma (h)=\mu \circ
\left(  \mu \otimes \alpha \right)\circ \left( f
\otimes g\otimes h \right)\circ \left( \Delta
\otimes \beta) \right) \Delta .
$$
Then, the Hom-associativity of $\mu$ and a Hom-coassociativity of
$\Delta$ lead to the Hom-associativity of the convolution product.
The unitality is as usual.
\end{proof}
\begin{definition}
 An endomorphism $S$ of $V$ is said to be
an\emph{ antipode} if it is  the inverse of the identity over $V$
for the Hom-algebra $Hom \left( V,V \right)$ with the multiplication
given by the convolution product defined by
$$ f \star g=\mu \circ
\left( f\otimes g \right) \Delta $$ and the unit
being $\eta \circ \epsilon$.
\end{definition}

The condition being antipode may be expressed by
the condition:
$$
\mu \circ S\otimes Id\circ \Delta = \mu \circ
Id\otimes S\circ \Delta =\eta \circ \varepsilon .
$$
\begin{definition}
A {\it Hom-Hopf algebra }is a Hom-bialgebra with an antipode.
\end{definition}
Then, a Hom-Hopf algebra over a $\mathbb{K}$-vector space $V$ is
given by $$\HH=(V,\mu ,\alpha,\eta ,\Delta, \beta ,\varepsilon ,S)$$

where the following homomorphisms
$$\mu :V\otimes V\rightarrow, \quad \eta :\mathbb{K}\rightarrow V,\quad \alpha :V\rightarrow V$$
$$\Delta :V\rightarrow V\otimes V, \quad \varepsilon :V\rightarrow \mathbb{K},\quad\beta :V\rightarrow V$$
$$S :V\rightarrow \mathbb{K}$$
satisfy the following conditions
\begin{enumerate}
\item $(V,\mu,\alpha ,\eta )$ is a unital  Hom-associative algebra.
\item $(V,\Delta, \beta ,\varepsilon)$ is a counital Hom-coalgebra.
\item $\Delta $ and $\varepsilon $ are morphisms of
algebras, which translate to
$$
\left\{
\begin{array}{l}
\Delta \left( e_1\right) =e_1\otimes e_1\qquad \text{where}\
e_1=\eta \left( 1\right) \\ \Delta \left( x\cdot y\right)=\Delta
\left( x \right) \bullet \Delta \left( y\right)
=\sum_{(x)(y)}x^{(1)}\cdot y^{\left( 1\right) }\otimes x^{\left(
2\right)
}\cdot y^{\left( 2\right) }\qquad \\
\varepsilon \left( e_1\right) =1 \\
\varepsilon \left( x\cdot y\right) =\varepsilon \left( x\right)
\varepsilon \left( y\right)
\end{array}
\right.
$$
\item $S$ is the antipode,  so
$$\mu \circ S\otimes Id\circ \Delta =\mu \circ Id\otimes S\circ \Delta =\eta
\circ \varepsilon .
$$
\end{enumerate}
\begin{remark} Let $V$ be a finite-dimensional $\K$-vector space.
If $H=(V,\mu ,\alpha,\eta ,\Delta ,
\beta,\varepsilon ,S)$ is a Hom-Hopf algebra,
then
$$H^*=(V^*,\Delta^* ,\beta^*,\varepsilon^*, \mu^*
, \alpha^*,\eta^* , S^*)$$ is also a Hom-Hopf
algebra.
\end{remark}

\subsection{Primitive elements and Generalized Primitive elements}

In the following, we discuss the properties of primitive elements in
a Hom-bialgebra.

Let  $\HH=(V,\mu,\alpha ,\eta ,\Delta, \beta ,\varepsilon)$ be a
Hom-bialgebra and $e_1=\eta (1)$ be the unit.

\begin{definition}
An element $x\in \HH$ is called primitive if $\Delta (x)=e_1\otimes
x + x\otimes e_1$.
\end{definition}
Let $x\in \HH$ be a primitive element, the coassociativity of
$\Delta$ implies
$$(\beta\otimes\Delta)\circ\Delta(x)=\tau_{13}\circ(\Delta\otimes\beta)\circ\Delta(x)
$$
where $\tau_{13}$ is a permutation in the symmetric group
$\mathcal{S}_3$.
\begin{lem}
Let $x$ be a primitive element in $\HH$, then $\varepsilon(x)=0$.
\end{lem}
\begin{proof}
By counity property, we have $ x=(id \otimes \varepsilon)\circ
\Delta (x)$. If $\Delta (x)=e_1\otimes x + x\otimes e_1$, then
$x=\varepsilon (x) e_1+\varepsilon (e_1) x$, and since $\varepsilon
(e_1)=1$ it implies $\varepsilon (x)=0$.
\end{proof}
\begin{proposition}

Let  $\HH=(V,\mu,\alpha ,\eta ,\Delta, \beta ,\varepsilon)$ be a
Hom-bialgebra and $e_1=\eta (1)$ be the unit.
 If $x$ and $y$ are two primitive elements in $\HH$. Then we have $\varepsilon (x)=0$ and  the commutator
 $[x,y]=\mu (x\otimes
y) -\mu (y \otimes x)$ is also a primitive element.

The set of all primitive elements of $\HH$,
denoted by $Prim(\HH)$,
has a structure of Hom-Lie algebra.
\end{proposition}
\begin{proof}
 By a direct calculation one has
\begin{eqnarray*}
   \Delta ([x,y]) &=& \Delta ( \mu (x\otimes y) -\mu (y \otimes x))\\
 \ &=& \Delta (x)\bullet \Delta (y)-\Delta (y)\bullet \Delta (x) \\
  \ &=& (e_1\otimes
x + x\otimes e_1)\bullet (e_1\otimes y + y\otimes e_1)- (e_1\otimes
y + y\otimes e_1)\bullet (e_1\otimes
x + x\otimes e_1)\\
 \ &=& e_1 \otimes \mu(x\otimes y)+y\otimes x+x\otimes y+\mu(x\otimes y)\otimes e_1\\ & &
 -e_1 \otimes \mu(y\otimes x )-x\otimes y-y\otimes x -\mu( y\otimes x)\otimes e_1\\
  \ &=& e_1 \otimes (\mu(x\otimes y)- \mu(y\otimes x ))+(\mu(x\otimes y)- \mu(y\otimes x ))\otimes e_1\\
  \ &=& e_1 \otimes
[x,y]+[x,y]\otimes e_1
\end{eqnarray*}

which means that $Prim(\HH)$ is closed under the bracket
multiplication $[\cdot,\cdot]$.

We have seen in \cite{MS} that there is a natural map from the
Hom-associative algebras to Hom-Lie algebras. The bracket $[x,y]=\mu
(x\otimes y) -\mu (y \otimes x)$ is obviously skewsymmetric and one
checks that the Hom-Jacobi condition is satisfied:

\begin{eqnarray*}
  [\alpha(x),[y,z]]-[[x,y],\alpha(z)]-[\alpha (y),[x,z]]=  \nonumber \\
   \mu (\alpha(x)\otimes \mu (y\otimes z))-\mu (\alpha(x)\otimes \mu (z\otimes y
   ))
   -\mu (\mu (y\otimes z)\otimes \alpha(x))+\mu (\mu (z\otimes y )\otimes \alpha(x))\nonumber \\
   -\mu (\mu (x\otimes y)\otimes \alpha(z)) +\mu (\mu (y\otimes x )\otimes \alpha(z))+
   \mu (\alpha(z)\otimes \mu (x\otimes y))-\mu (\alpha(z)\otimes \mu (y\otimes
   x
   ))\nonumber \\
   -\mu (\alpha(y)\otimes \mu (x\otimes z))+\mu (\alpha(y)\otimes \mu (z\otimes
   x
   ))
   +\mu (\mu (x\otimes z)\otimes \alpha(y))-\mu (\mu (z\otimes x )\otimes
   \alpha(y))=0
\end{eqnarray*}

\end{proof}

We introduce now a notion of generalized primitive element.
\begin{definition}
An element $x\in \HH$ is called generalized primitive element if it
satisfies the conditions
\begin{equation}
\label{GPrim}(\beta\otimes\Delta)\circ\Delta(x)
=\tau_{13}\circ(\Delta\otimes\beta)\circ\Delta(x)
\end{equation}
\begin{equation}
\Delta^{op}(x)=\Delta(x)
\end{equation}

where $\tau_{13}$ is a permutation in the symmetric group
$\mathcal{S}_3$.
\end{definition}
\begin{remark}
\begin{enumerate}
\item In particular, a primitive element in $\HH$ is a generalized
primitive element.
\item The condition \eqref{GPrim} may be written
$$(\Delta\otimes\beta)\circ\Delta(x)=
\tau_{13}\circ(\beta\otimes\Delta)\circ\Delta(x).
$$
\end{enumerate}
\end{remark}
\begin{proposition}
Let  $\HH=(V,\mu,\alpha ,\eta ,\Delta, \beta ,\varepsilon)$ be a
Hom-bialgebra and $e_1=\eta (1)$ be the unit.
 If $x$ and $y$ are two generalized primitive elements in $\HH$.
 Then, we have $\varepsilon (x)=0$ and the commutator $[x,y]=\mu (x\otimes y) -\mu (y
\otimes x)$ is also a generalized primitive element.

The set of all generalized primitive elements of $\HH$, denoted by
$GPrim(\HH)$, has a structure of Hom-Lie algebra.
\end{proposition}
\begin{proof}
Let $x$ and $y$ be two generalized primitive
elements in $\HH$. In the following the
multiplication $\mu$ is denoted by a dot. The
following equalities hold:
\begin{eqnarray*}
  (\Delta\otimes\beta)\circ\Delta(x\cdot y-y\cdot x) &=& (\Delta\otimes\beta)\circ\Delta(x\cdot y)-
  (\Delta\otimes\beta)\circ\Delta(y\cdot x) \\
  \ &=& (\Delta\otimes\beta)(\Delta(x)\bullet \Delta(y))- (\Delta\otimes\beta)(\Delta(y)\bullet \Delta(x))\\
 \ &=& \Delta(x^{(1)}\cdot y^{(1)})\otimes \beta(x^{(2)}\cdot y^{(2)})-
 \Delta(y^{(1)}\cdot x^{(1)})\otimes \beta(y^{(2)}\cdot x^{(2)})\\
  \ &=& (x^{(1)(1)}\cdot y^{(1)(1)})\otimes(x^{(1)(2)}\cdot y^{(1)(2)})\otimes \beta(x^{(2)}\cdot
  y^{(2)})\\ && -(y^{(1)(1)}\cdot x^{(1)(1)})\otimes(y^{(1)(2)}\cdot x^{(1)(2)})\otimes \beta(y^{(2)}\cdot
  x^{(2)}).
\end{eqnarray*}
Then, using the fact that $\Delta^{op}=\Delta$ for generalized
primitive elements one has:
\begin{eqnarray*}
  \tau_{13}\circ(\Delta\otimes\beta)\circ\Delta(x\cdot y-y\cdot x) &=&\beta(x^{(2)}\cdot
  y^{(2)})\otimes(x^{(1)(2)}\cdot y^{(1)(2)})\otimes
  (x^{(1)(1)}\cdot y^{(1)(1)}) \\
  &&-\beta(y^{(2)}\cdot
  x^{(2)})\otimes(y^{(1)(2)}\cdot x^{(1)(2)})\otimes (y^{(1)(1)}\cdot x^{(1)(1)})  \\
  \ &=& (\beta\otimes\Delta)
  \circ\Delta(x\cdot y-y\cdot x).
\end{eqnarray*}
The structure of Hom-Lie algebra follows from the same argument as
in the primitive elements case.
\end{proof}
\subsection{Antipode's properties}
 Let $\HH=\left( V,\mu,\alpha ,\eta ,\Delta ,\beta,\varepsilon, S
\right) $  be a Hom-Hopf algebra.

For any element $x\in V$, using the counity and Sweedler notation,
one may write
\begin{equation} x=\sum_{(x)}{x^{(1)}\otimes
\varepsilon (x^{(2)})}= \sum_{(x)}{\varepsilon
(x^{(1)})\otimes x^{(2)}}. \label{Coun}
\end{equation}
Then, for any $f\in End_\K (V)$,
we have
\begin{equation}
f(x)=\sum_{(x)}{f(x^{(1)}) \varepsilon
(x^{(2)})}= \sum_{(x)}{\varepsilon
(x^{(1)})\otimes f(x^{(2)})}. \label{Coun2}
\end{equation}
Let $ f \star g=\mu \circ \left( f\otimes g \right) \Delta $ be the
convolution product of $f,g\in End_\K(V)$.
One may write
\begin{equation}
(f\star g)(x)= \sum_{(x)}{\mu ( f(x^{(1)})\otimes
g (x^{(2)}))}. \label{starpro}
\end{equation}
Since the antipode $S$ is the inverse of the identity for the
convolution product then $S$ satisfies
\begin{equation}
\varepsilon(x)\eta (1)= \sum_{(x)}{\mu (
S(x^{(1)})\otimes x^{(2)})}=\sum_{(x)}{\mu (
x^{(1)}\otimes S(x^{(2)})}). \label{anti1}
\end{equation}

\begin{proposition}
The antipode $S$ is unique and we have
\begin{itemize}
\item  S($\eta (1))=\eta (1).$
\item  $\varepsilon \circ S=\varepsilon$.
\end{itemize}
\end{proposition}
\begin{proof}
1) We have $S\star id=id \star S=\eta \circ
\varepsilon$. Thus, $(S \star id) \star S=S \star
(id \star S)=S$. If $S'$ is another antipode of
$\HH$ then
$$S'=S' \star id \star S'=S' \star id
\star S=S \star id \star S=S.$$ Therefore the
antipode when it exists is unique.

2) Setting $e_1=\eta (1)$ and since $\Delta
(e_1)= e_1 \otimes e_1$ one has
$$(S\star id) (e_1)=
\mu (S(e_1)\otimes e_1) =S(e_1)=\eta(\varepsilon
(e_1))=e_1.$$

3) Applying \eqref{Coun2} to $S$, we obtain $
S(x)=\sum_{(x)}{S(x^{(1)}) \varepsilon
(x^{(2)})}$.

Applying $\varepsilon$ to \eqref{anti1}, we
obtain
$$\varepsilon (x)=\varepsilon
(\sum_{(x)}\mu(S(x^{(1)})\otimes x^{(2)})).$$
Since $\varepsilon$ is a Hom-algebra morphism,
one has
$$\varepsilon (x)=\sum_{(x)}\varepsilon(S(x^{(1)}))\varepsilon
(x^{(2)})=\varepsilon(\sum_{(x)}S(x^{(1)})
\varepsilon(x^{(2)}))=\varepsilon(S(x)).
$$
Thus  $\varepsilon \circ S=\varepsilon$.
\end{proof}

\subsection{ Modules and  Comodules}

We introduce in the following the structure of module and comodule
over Hom-associative algebras.

Let $\A=(V,\mu,\alpha)$ be a Hom-associative
$\K$-algebra, an $\A$-module (left) is a triple
$(M,f,\gamma)$ where $M$ is $\K$-vector space and
$f,\gamma$ are  $\K$-linear maps, $f:  M
\rightarrow M$ and $\gamma : V \otimes M
\rightarrow M$, such that the following diagram
commutes:

$$
\begin{array}{ccc}
V\otimes V\otimes M & \stackrel{\mu \otimes f}{\longrightarrow } &
V\otimes
M \\
\ \quad \left\downarrow ^{\alpha \otimes \gamma }\right. &  & \quad
\left\downarrow
^\gamma \right. \\
V\otimes M & \stackrel{\gamma }{\longrightarrow } & M
\end{array}
$$
The dualization leads to  comodule definition
over a Hom-coassociative coalgebra.

Let  $C=(V,\Delta, \beta)$ be a Hom-coassociative
coalgebra. A $C$-comodule (right) is a triple
$(M,g,\rho)$ where $M$ is a $\K$-vector space and
$g,\rho$ are $\K$-linear maps, $g: M \rightarrow
M$ and $\rho : M \rightarrow M\otimes V $, such
that the following diagram commutes:

$$
\begin{array}{ccc}
 M & \stackrel{\rho}{\longrightarrow } & M \otimes V
 \\
\ \quad \left\downarrow ^{\rho}\right. &  & \quad \left\downarrow
^{g \otimes \Delta} \right. \\
M\otimes V & \stackrel{\rho\otimes \beta}{\longrightarrow } &
M\otimes V\otimes V
\end{array}
$$
\begin{remark}
A Hom-associative $\K$-algebra
$\A=(V,\mu,\alpha)$ is  a left $\A$-module with
$M=V$, $f=\alpha$ and $\gamma =\mu$. Also, a
Hom-coassociative coalgebra $C=(V,\Delta, \beta)$
is a right  $C$-comodule with $M=V$, $g=\beta$
and $\rho =\Delta$. The properties of modules and
comodules over Hom-associative algebras or
Hom-coassociative algebras will be discussed in a
forthcoming paper.
\end{remark}

\subsection{Examples}
The classification of $2$-dimensional
Hom-associative algebras, up to isomorphism,
yields to the following two classes. Let
$B=\{e_1,e_2\}$ be a basis where $\eta(1)=e_1$ is
the unit.
\begin{enumerate}
\item
The multiplication $\mu_1$ is defined by $\mu_1(e_1\otimes
e_i)=\mu_1(e_i\otimes e_1)=e_i$ for $i=1,2$ and $\mu_1(e_2\otimes
e_2)=e_2$ and the homomorphism $\alpha_1$ is defined, with respect
to the basis $B$ by $\left(
                    \begin{array}{cc}
                      a_1 & 0 \\
                      a_2-a_1 & a_2 \\
                    \end{array}
                  \right).$
\item
The multiplication $\mu_2$ is defined by $\mu_2(e_1\otimes
e_i)=\mu_2(e_i\otimes e_1)=e_i$ for $i=1,2$ and $\mu_2(e_2\otimes
e_2)=0$ and the homomorphism $\alpha_2$ is defined, with respect to
the basis $B$ by $\left(
                    \begin{array}{cc}
                      a_1 &0  \\
                      a_2 & a_1 \\
                    \end{array}
                  \right).$
\end{enumerate}
The Hom-bialgebras corresponding to the Hom-associative algebra
defined by $\mu_1$ and $\alpha_1$ are given in the following table \
\vspace{0.5cm}
\begin{small}
\begin{center}
\begin{tabular}{|c|c|c|c|}
  \hline
  \ & Comultiplication & Co-unit& homomorphism\\\hline
  1 & $\begin{array}{c}
           \Delta (e_1)= e_1\otimes e_1 \\
           \Delta (e_2)= e_2\otimes e_2 \\
         \end{array}$
   & $ \begin{array}{c}
          \varepsilon (e_1)= 1 \\
           \varepsilon (e_2)= 1 \\
         \end{array}$ & $ \left(
                    \begin{array}{cc}
                      b_1 &0  \\
                       b_3& b_2 \\
                    \end{array} \right) $\\ \hline
  2 & $ \begin{array}{c}
           \Delta (e_1)= e_1\otimes e_1 \\
           \Delta (e_2)= e_1\otimes e_2+e_2\otimes e_1-2 e_2\otimes e_2 \\
         \end{array}$ & $\begin{array}{c}
          \varepsilon (e_1)= 1 \\
           \varepsilon (e_2)= 0 \\
         \end{array}$ & $ \left(
                    \begin{array}{cc}
                      b_1 & \frac{b_1-b_3}{2} \\
                      b_2 & b_3 \\
                    \end{array} \right) $\\ \hline
  3 &$ \begin{array}{c}
           \Delta (e_1)= e_1\otimes e_1 \\
           \Delta (e_2)= e_1\otimes e_2+e_2\otimes e_1- e_2\otimes e_2 \\
         \end{array} $& $ \begin{array}{c}
          \varepsilon (e_1)= 1 \\
           \varepsilon (e_2)= 0 \\
         \end{array} $& $ \left(
                    \begin{array}{cc}
                      b_1 & b_1-b_3 \\
                      b_2 & b_3 \\
                    \end{array} \right)  $\\ \hline
\end{tabular}
\end{center}
\end{small}

\vspace{0.5cm}

Only Hom-bialgebra (2)  carries a structure of Hom-Hopf algebra with
an antipode defined, with respect to a basis $B$, by the identity
matrix.
\begin{remark}
There is no Hom-bialgebra associated to the Hom-associative algebra
defined by the multiplication $\mu_2$ and any homomorphism
$\alpha_2$.
\end{remark}


\begin{thebibliography}{99}
\bibitem{Drinfeld}  Drinfel'd V. G., \textit{Hopf algebras and the quantum Yang-Baxter equation, } Soviet Math. Doklady, 32
 (1985), pp  254-258.
\bibitem{GR}  Goze M., Remm E.,
\textit{Lie-admissible coalgebras }, J. Gen. Lie
Theory Appl. {\bf 1} (2007), no. 1, 19--28.
\bibitem{Guichardet}  Guichardet A., \textit{groupes quantiques }InterEditions / CNRS Editions (1995).
\bibitem{HLS} Hartwig J. T., Larsson D., Silvestrov S. D.,
\emph{Deformations of Lie algebras using $\sigma$-derivations}, J.
Algebra \textbf{295} (2006), pp 314-361.
\bibitem{HS} Hellstr\"{o}m L.,  Silvestrov S. D.,
\emph{Commuting elements in $q$-deformed Heisenberg algebras}, World
Scientific, (2000).
\bibitem{Kassel}  Kassel C., \textit{Quantum groups } Spriner-Verlag, Graduate Text in Mathematics (1995).
\bibitem{LS1} Larsson D., Silvestrov S. D.,
\emph{Quasi-hom-Lie algebras, Central Extensions and 2-cocycle-like
identities}, J. Algebra \textbf{288} (2005), pp 321--344.
\bibitem{LS2} Larsson D., Silvestrov S. D.,
\emph{Quasi-Lie algebras}, in "Noncommutative
Geometry and Representation Theory in Mathematical Physics",
Contemp. Math., \textbf{391}, Amer. Math. Soc., Providence, RI,
(2005), pp 241-248.
\bibitem{LS3} Larsson D., Silvestrov S. D.,
\emph{Quasi-deformations of $sl_2(\mathbb{F})$
using twisted derivations}, to appear in Comm.
Algebra. \\ (arXiv.org:\verb|math.RA/0506172|).
\bibitem{Majid}  Majid S., \textit{Foundations of quantum group theory }Cambridge University Press (1995).
\bibitem{Makhlouf-Hopf}  Makhlouf A.,
\textit{ Degeneration, rigidity and irreducible components of Hopf
algebras,}  Algebra Colloquium, vol \textbf{12} (2) (2005), pp
241-254.
\bibitem{MS} Makhlouf A., Silvestrov S. D.,
\emph{Hom-algebra structures}, to appear in J.
Gen. Lie Theory Appl. \\ (Preprints in
Mathematical Sciences, Lund University, Centre
for Mathematical Sciences, Centrum Scientiarum
Mathematicarum, (2006:10) LUTFMA-5074-2006;
arxiv.org/math/0609501v3)
\bibitem{MontgomeryLivre}  Montgomery S., \textit{Hopf algebras and their actions on rings, } AMS Regional Conference Series in Mathematics,
\textbf{82}, (1993).
\bibitem{Yau:EnvLieAlg} Yau D.,
\emph{Enveloping algebra of Hom-Lie algebras}, to
appear in J. Gen. Lie Theory Appl.
\end{thebibliography}
\end{document}